\documentclass[10pt,a4paper]{article}
\usepackage[latin1]{inputenc}
\usepackage{amsmath}
\usepackage{amsfonts}
\usepackage{amssymb}
\usepackage{graphicx}
\author{
	Stefano Baratella\footnote{
Dipartimento di Matematica,
Universit{\`a} di Trento, Italy.
E-mail:~\textsf{stefano.baratella@unitn.it}
}
	 \and Andrea Masini\footnote{
 	Dipartimento di Informatica,
 	Universit{\`a} di Verona, Italy.
 	E-mail:~\textsf{andrea.masini@univr.it}
	 	}
}
\title{A two-dimensional metric temporal logic
\footnote{\textsf{An updated version of this paper will appear in: Mathematical Logic Quarterly}}}
\date{}
\usepackage{latexsym}

\usepackage{calc,sidecap,caption2,amsmath,amssymb,amsthm,color,cancel} 

\theoremstyle{plain}
     \newtheorem{theorem}{Theorem}[section] 
     
     \newtheorem{corollary}[theorem]{Corollary}
     
\theoremstyle{definition}
      
\theoremstyle{remark}
     \newtheorem{remark}{Remark}[section]

\DeclareSymbolFont{FormalScript}{U}{rsfs}{m}{n}
\DeclareSymbolFontAlphabet\mathscript{FormalScript}

\mathchardef\semicolon="603B 
\mathchardef\gt="313E
\mathchardef\lt="313C

\newcount\mscount
\newdimen\proofrulebreadth \proofrulebreadth=.05em
\newdimen\proofdotseparation \proofdotseparation=1.25ex
\newdimen\proofrulebaseline \proofrulebaseline=2ex
\newcount\proofdotnumber \proofdotnumber=3
\let\then\relax
\def\hfi{\hskip0pt plus.0001fil}
\mathchardef\squigto="3A3B
\newif\ifinsideprooftree\insideprooftreefalse
\newif\ifonleftofproofrule\onleftofproofrulefalse
\newif\ifproofdots\proofdotsfalse
\newif\ifdoubleproof\doubleprooffalse
\let\wereinproofbit\relax
\newdimen\shortenproofleft
\newdimen\shortenproofright
\newdimen\proofbelowshift
\newbox\proofabove
\newbox\proofbelow
\newbox\proofrulename
\def\shiftproofbelow{\let\next\relax\afterassignment\setshiftproofbelow\dimen0 }
\def\shiftproofbelowneg{\def\next{\multiply\dimen0 by-1 }%
\afterassignment\setshiftproofbelow\dimen0 }
\def\setshiftproofbelow{\next\proofbelowshift=\dimen0 }
\def\setproofrulebreadth{\proofrulebreadth}

\def\prooftree{
\ifnum	\lastpenalty=1
\then	\unpenalty
\else	\onleftofproofrulefalse
\fi
\ifonleftofproofrule
\else	\ifinsideprooftree
	\then	\hskip.5em plus1fil
	\fi
\fi
\bgroup
\setbox\proofbelow=\hbox{}\setbox\proofrulename=\hbox{}%
\let\justifies\proofover\let\leadsto\proofoverdots\let\Justifies\proofoverdbl
\let\using\proofusing\let\[\prooftree
\ifinsideprooftree\let\]\endprooftree\fi
\proofdotsfalse\doubleprooffalse
\let\thickness\setproofrulebreadth
\let\shiftright\shiftproofbelow \let\shift\shiftproofbelow
\let\shiftleft\shiftproofbelowneg
\let\ifwasinsideprooftree\ifinsideprooftree
\insideprooftreetrue
\setbox\proofabove=\hbox\bgroup$\displaystyle 
\let\wereinproofbit\prooftree
\shortenproofleft=0pt \shortenproofright=0pt \proofbelowshift=0pt
\onleftofproofruletrue\penalty1
}

\def\eproofbit{
\ifx	\wereinproofbit\prooftree
\then	\ifcase	\lastpenalty
	\then	\shortenproofright=0pt	
	\or	\unpenalty\hfil		
	\or	\unpenalty\unskip	
	\else	\shortenproofright=0pt	
	\fi
\fi
\global\dimen0=\shortenproofleft
\global\dimen1=\shortenproofright
\global\dimen2=\proofrulebreadth
\global\dimen3=\proofbelowshift
\global\dimen4=\proofdotseparation
\global\mscount=\proofdotnumber
$\egroup  
\shortenproofleft=\dimen0
\shortenproofright=\dimen1
\proofrulebreadth=\dimen2
\proofbelowshift=\dimen3
\proofdotseparation=\dimen4
\proofdotnumber=\mscount
}

\def\proofover{
\eproofbit 
\setbox\proofbelow=\hbox\bgroup 
\let\wereinproofbit\proofover
$\displaystyle
}%
\def\proofoverdbl{
\eproofbit 
\doubleprooftrue
\setbox\proofbelow=\hbox\bgroup 
\let\wereinproofbit\proofoverdbl
$\displaystyle
}%
\def\proofoverdots{
\eproofbit 
\proofdotstrue
\setbox\proofbelow=\hbox\bgroup 
\let\wereinproofbit\proofoverdots
$\displaystyle
}%
\def\proofusing{
\eproofbit 
\setbox\proofrulename=\hbox\bgroup 
\let\wereinproofbit\proofusing
\kern0.3em$
}

\def\endprooftree{
\eproofbit 
  \dimen5 =0pt
\dimen0=\wd\proofabove \advance\dimen0-\shortenproofleft
\advance\dimen0-\shortenproofright
\dimen1=.5\dimen0 \advance\dimen1-.5\wd\proofbelow
\dimen4=\dimen1
\advance\dimen1\proofbelowshift \advance\dimen4-\proofbelowshift
\ifdim	\dimen1<0pt
\then	\advance\shortenproofleft\dimen1
	\advance\dimen0-\dimen1
	\dimen1=0pt
	\ifdim  \shortenproofleft<0pt
        \then   \setbox\proofabove=\hbox{%
			\kern-\shortenproofleft\unhbox\proofabove}%
                \shortenproofleft=0pt
        \fi
\fi
\ifdim	\dimen4<0pt
\then	\advance\shortenproofright\dimen4
	\advance\dimen0-\dimen4
	\dimen4=0pt
\fi
\ifdim	\shortenproofright<\wd\proofrulename
\then	\shortenproofright=\wd\proofrulename
\fi
\dimen2=\shortenproofleft \advance\dimen2 by\dimen1
\dimen3=\shortenproofright\advance\dimen3 by\dimen4
\ifproofdots
\then
	\dimen6=\shortenproofleft \advance\dimen6 .5\dimen0
	\setbox1=\vbox to\proofdotseparation{\vss\hbox{$\cdot$}\vss}
	\setbox0=\hbox{%
		\kern\dimen6
		$\vcenter to\proofdotnumber\proofdotseparation
			{\leaders\box1\vfill}$%
		\unhbox\proofrulename}%
\else	\dimen6=\fontdimen22\the\textfont2 
	\dimen7=\dimen6
	\advance\dimen6by.5\proofrulebreadth
	\advance\dimen7by-.5\proofrulebreadth
	\setbox0=\hbox{%
		\kern\shortenproofleft
		\ifdoubleproof
		\then	\hbox to\dimen0{%
			$\mathsurround0pt\mathord=\mkern-6mu%
			\cleaders\hbox{$\mkern-2mu=\mkern-2mu$}\hfill
			\mkern-6mu\mathord=$}%
		\else	\vrule height\dimen6 depth-\dimen7 width\dimen0
		\fi
		\unhbox\proofrulename}%
	\ht0=\dimen6 \dp0=-\dimen7
\fi
\let\doll\relax
\ifwasinsideprooftree
\then	\let\VBOX\vbox
\else	\ifmmode\else$\let\doll=$\fi
	\let\VBOX\vcenter
\fi
\VBOX	{\baselineskip\proofrulebaseline \lineskip.2ex
	\expandafter\lineskiplimit\ifproofdots0ex\else-0.6ex\fi
	\hbox	spread\dimen5	{\hfi\unhbox\proofabove\hfi}%
	\hbox{\box0}%
	\hbox	{\kern\dimen2 \box\proofbelow}}\doll%
\global\dimen2=\dimen2
\global\dimen3=\dimen3
\egroup 
\ifonleftofproofrule
\then	\shortenproofleft=\dimen2
\fi
\shortenproofright=\dimen3
\onleftofproofrulefalse
\ifinsideprooftree
\then	\hskip.5em plus 1fil \penalty2
\fi
}

\newcommand{\urule}[3]{%
	\prooftree #1 \justifies #2 \using #3 \endprooftree}
\newcommand{\brule}[4]{%
	\prooftree #1\ \ \ #2 \justifies #3 \using #4 \endprooftree}

\def\TR{\mathbb{T}}

\begin{document}
	\maketitle
	\begin{abstract}
		We introduce  a two-dimensional metric (interval) temporal logic whose internal and external time flows are dense linear orderings.  We provide a suitable semantics  and a sequent calculus with axioms for equality and extralogical axioms. Then we prove completeness  and a semantic partial cut-elimination theorem down to formulas of a certain type. 
		
	\end{abstract}

	
	\section{Introduction} 
	
	In recent years, metric temporal logics gained popularity because of their applicative aspects, for instance those relative to the formalization of time-critical systems. At the same time, under the paradigm of \textit{time granularity},  there appeared in the literature a number of proposals on how to combine temporal logics into $n$-dimensional systems (very often  with $n=2$). See, for instance, \cite{fg1,fg2,r98}.
	To the best of our knowledge,  proof-theoretic properties of those multi-dimensional systems have not been investigated yet.

	In computer science,  multi-dimensional temporal structures and the corresponding logics are often  used. For instance, there are  temporal structures modelling computations, with internal and external temporal quantifications ranging on states and  computations respectively. There are also  the structures for time granularity, which come equipped with an internal time flow (for each of the so called time grains) and an external one (for the whole structure of time grains).  
	In this paper we took inspiration from the latter structures to develop a corresponding two-dimensional metric temporal logic, where both time flows are dense linear orderings.  First of all we establish
	completeness of the  proposed system with respect to a suitable class of structures. Secondly we get   a semantic proof of cut-elimination  as a byproduct of the technique used to establish completeness. More specifically, regarding cut-elimination, we aim at a  semantic  partial cut-elimination result in the vein of \cite[Theorem 2.7.1]{ptlc}. We get our result by using  semantic techniques, as done in \cite[3.1.9]{ptlc}. We show that every provable sequent has a proof whose cut formulas are among those occurring  in some extralogical axiom.   
	
	We recall that, in presence of extralogical axioms,  a full cut-elimination result that preserves the subformula property does not hold in general. See \cite[Remark 2.7.7 ]{ptlc} for an example.

	For simplicity, in the following  we will simply refer to our result  as to cut-elimination.

	\medskip
	
	In \cite{bm} we introduced an infinitary extension of a fragment of the Metric Interval  Temporal Logic  over dense time domains, called ${\rm MTL}_\infty.$   We recall that Metric Interval  Temporal Logic was introduced in \cite{mitl} as a fragment of Metric Temporal Logic in which  
	temporal operators may ``quantify''   over nonsingular intervals only.
	More precisely, the logic 
	${\rm MTL}_\infty$  is a modification and an extension of a system proposed in \cite{mpt}.  Same as with the latter system,  only quantifications over nonsingular intervals are allowed in ${\rm MTL}_\infty$. Differently from the system in \cite{mpt}, an induction schema is provable in ${\rm MTL}_\infty,$ thanks to the presence of some infinitary rules and axioms. Moreover,  applications  of the cut rule can be restricted to very simple formulas (see \cite[Theorem 3.4]{bm}).  In ${\rm MTL}_\infty$   there are relational formulas that describe properties of the underlying time flow and labelled formulas that express temporal statements. Regarding labelled deductive systems, we refer the reader to  \cite{g96,v00}.
	
	In ${\rm MTL}_\infty,$ binary propositional connectives only apply to formulas of the same kind and the deduction rules deal separately (as much as possible) with the relational and the labelled component of the deductive system. 
	
	Apart from recovering an induction schema and retaining  a suitable cut-elimination property, a motivation for studying ${\rm MTL}_\infty$ was the remark made in   \cite{mpt}   that,  differently from the fully developed
	model checking techniques, proof-theoretic investigations of Metric Temporal Logic had  been only partly attempted until that time.
	
	Although familiarity with the  content of \cite{bm} might be helpful, we will keep this paper as much self-contained as possible. We will refer to \cite{bm}  to avoid unnecessary repetitions or  for  comments and further motivations for introducing ${\rm MTL}_\infty$. 
	
	In this paper we aim at temporalizing ${\rm MTL}_\infty.$ More precisely, but still vaguely,  we want to put  a copy of ${\rm MTL}_\infty$ on top of itself and investigate the properties of the  resulting two-dimensional system, to be called ${\rm MTL}_\infty^2.$ The semantic intuition is that there are  two kinds of time flows, an internal one and an external one, each being a dense linear ordering with least but no greatest element.  At  each point of the external flow it corresponds a copy of the internal time flow. Roughly speaking, each copy of the internal time flow and the external one will be governed by the ${\rm MTL}_\infty$ logic of \cite{bm}. Temporal operators do  quantify over (possibly unbounded) intervals in the same way as in ${\rm MTL}_\infty.$
	
	It is worth mentioning that the temporalization of certain systems has been  studied in  \cite{fg1}, where the authors  investigate the so called \textit{external} way of temporalizing a logic system. In the external approach it is not necessary to have detailed knowledge about the components of the system. We anticipate that, differently from  \cite{fg1}, we will need a detailed knowledge of the components  in order to establish cut elimination for the resulting system. 
	
	Various techniques for combining  temporal logic systems have been developed also   in \cite{fg2}. In that approach, constraints need to be specified in order to ensure that properties of the combining logics are retained by the combined system.   More precisely, in \cite{fg2} the authors investigate the transfer of soundness, completeness and decidability from the components to the system they generate. We point out that   they do not establish  any proof-theoretic property of the resulting systems. In this paper, in addition to proving soundness and completeness  of  ${\rm MTL}_\infty^2$  with respect to adequate semantics and sequent calculus,  we will also get a  semantic proof of a partial cut-elimination theorem.

	We  also point out that  a two-dimensional  system can be used to describe a temporal logic for time granularity.    For such a logic  can be regarded as a  combination of simpler  temporal logics in a way that properties of the resulting system can be derived from those of its components.  For a survey on temporal logics for time granularity, we refer the reader  to \cite{em}. 
	
	Eventually we recall that,  in the literature, one can find general methods for  obtaining a system that admits full cut-elimination, starting from a system with  axioms. They are based on the replacement of   axioms with rules.  See the method introduced in   \cite{NvP} and subsequently developed  by the same authors  in other works.  In our opinion, the technique proposed in  \cite{NvP} is a clever way of hiding the cuts within the added rules. The reader should bear in mind that, differently from our work, \cite{NvP} originates from   purely proof-theoretical motivations.  A stronger motivation for  dealing with a system which is not completely cut-free is that  application of the technique  in  \cite{NvP} to a concrete deduction system with axioms like ours  results in the replacement of  naturally formulated  axioms  (whose intuitive meaning is clearly understood) with rules which are far less intuitive.
	In our opinion, the latter phenomenon is not a point in favor of a potential usability of the resulting  system.
	
	\medskip 
	
	The rest of the paper is organized as follows.

	In Section~\ref{sinsem},  we introduce syntax and  semantics of ${\rm MTL}_\infty^2.$ 
	
	In Section~\ref{cal} we introduce the sequent calculus. We also {   point out}  that  two induction 
	schemas (one for each kind of  time flow) are provable in ${\rm MTL}_\infty^2.$
	
	In Section~\ref{compl} we establish completeness and we get  cut-elimination, by suitably modifying  the reduction tree construction technique of predicate logic. Differently from \cite{bm}, we do not make use of any first order translation of the  ${\rm MTL}_\infty^2$ temporal formulas, but we directly construct the reduction tree of a sequent by means of the ${\rm MTL}_\infty^2$ deduction rules.  
	
	Eventually, { by applying a  result  in \cite{bm}}, we get completeness of ${\rm MTL}_\infty^2$  with respect to  { the family of  structures where the  underlying time flows are given by the nonnegative rational numbers and the time intervals are determined by a very simple function.}
	
	\section{Syntax and  semantics}\label{sinsem}
	
	We recall that the logic $L_{\omega_1\omega}$ is an extension of first order logic where  countable
	conjunctions and disjunctions are allowed (see \cite{keisler}). 
	
	As is usual we  identify a language $L$ with its extra logical symbols. We write $\varphi\in L$ as an abbreviation for ``$\varphi$ is an $L$-formula''. Similarly with $L$-terms. 
	
	We denote the set of natural numbers by $\omega.$
	
	\subsection{Syntax}\label{syntax}
	
	The language of ${\rm MTL}^2_\infty$ is formed by an internal language $L_i$ and an external language $L_e.$  Each of these two languages comprises a   \textit{relational} component  $L_{*, r}.$   and a  \textit{labelled} component $L_{*,l},$ where $*\in\{i,e\}.$ We begin by introducing  the two components of  $L_i$ and the corresponding formulas.
	
	\begin{itemize}
		\item $L_{i, r}=\{ f, <, c\}$ is a first order language with equality and $f$, $<$, $c$ are a unary function, a binary relation and a constant symbol respectively. 
		We have one additional connective $\bigvee$ for (countable) infinitary disjunction.  We use lower case letters $v, w, x, y, z$ (possibly indexed) for the (meta)variables of  $L_{i, r}.$ We denote by $V_i$ the set  of $L_{i, r}$-variables.

		\medskip We denote by $T_i$ the set  of  $L_{i, r}$-terms. We use letters $s,t, u$ (possibly indexed) for the $L_{i, r}$-terms.
		The set $F_{i, r}$ of $L_{i, r}$-formulas (to be called \textit{internal relational formulas}) is the least set 
		containing:
		\begin{enumerate}
			\item[--] the finitary $L_{i, r}$-formulas;
			\item[--] all the subformulas (in the sense of the $L_{\omega_1\omega}$ logic) of
			the infinitary formulas of the form $\forall x\bigvee_{n\in\omega}(x\lt
			f^n(c)).$ Hence the set of subformulas of $\forall x\bigvee_{n\in\omega}(x\lt
			f^n(c))$ contains the formula itself, $\bigvee_{n\in\omega}(t\lt
			f^n(c))$  and $t\lt
			f^k(c)$ for all relational terms $t$ and all $k\in{\omega}.$ (Here and in the following we let $f^0(t)=t$ for all terms $t.$)
		\end{enumerate}

		\item $L_{i, t}$ is a propositional language with a set $\mathbf P=\{ p_n: n\in\omega\}$ of propositional letters; the propositional connectives and, for all  $m<n$ in ${\omega },$ the temporal operators
		$$\Box_{[m,n]}\ \quad\Box_{]m,n]}\ \quad \Box_{[m,n[}\quad \Box_{]m,n[}\ \quad
		\Box_{[m,\infty[}\ \quad  \Box_{]m,\infty[}\ .$$

		The semantics of the temporal operators will be introduced in Section~\ref{semantics}. 
		We write $\Box_{(m,n)}$ to denote any of the above  operators, where $m\in\omega$ and $n\in\omega\cup\{\infty\}.$
		
		\medskip
		
		The set $F_{i, t}$ of $L_{i, t}$-formulas (to be called \textit{internal temporal formulas}) is the least set $A$  that contains the proposition symbols and is closed under application of the propositional connectives  and under the following formation rules: if $\alpha\in A $ then 
		$\Box_{(m,n)}\alpha$ is in $A,$ where $\Box_{(m,n)}$ is any of the temporal operators.
		
	\end{itemize}

	The set $F_{i, l}$ of \textit{internal labelled formulas} is the set of all expressions of the form $s: \psi,$ where $s\in T_i$ and $\psi\in F_{i, t}.$ 
	
	\medskip
	
	What we have introduced so far is basically the syntax of ${\rm MTL}_\infty.$ Logic ${\rm MTL}_\infty$ has two kinds of formulas:  the relational  and the labelled ones.  To sum up: the  formulas in $F_{i,t}$ are those of the propositional metric temporal logic  underlying ${\rm MTL}_\infty.$  A formula $\alpha\in F_{i,t}$  can be decorated by an internal label $s$ (an element of $T_i$) to form the internal labelled formula $s:\alpha.$ Once $s:\alpha$ has been formed, it cannot be used to construct a more complicated formula. By this we mean that, for instance, the conjunction of two internal labelled formulas is not an internal labelled formula, but it will be a formula at the second level of the temporalization, as we explain below.
	
	As we said already,  ${\rm MTL}^2_\infty$ is obtained by putting  a copy of ${\rm MTL}_\infty$ on top ${\rm MTL}_\infty$ itself. Therefore we define $L_{e, r}, T_e, F_{e,r}, $ (``e'' for ``external'') as above, by using the same symbols, this time  in upper case, for the extralogical symbols and  the variables. There will be no ambiguity in using $<$ for both the internal and the external predicate symbols. 
	
	Let $\overline F_{i,l}$ be the the set obtained by closing  $F_{i,l}$ under application of (external) propositional connectives and temporal operators.   The intuition is that  the formulas in $F_{i,l}$ play the role of the atomic temporal formulas for the external logic and $\overline F_{i,l}$ is thus the set of all the external temporal formulas.  The set $F_{e, l}$ of \textit{external labelled formulas} is the set of all expressions of the form $S: \beta,$ where $S\in T_e$ and $\beta\in \overline F_{i,l}\cup F_{i,r}.$  The definition of $F_{e,l}$ agrees with the spirit of ${\rm MTL}_\infty,$ where  the relational and the labelled formulas are not ``mixed''.  We will use upper case letters $A, B, \dots$ (possibly indexed)  to denote formulas in $F_{e,r}\cup F_{e,l}.$  We say that a formula  $S:\beta$ in $F_{e,l}$ is atomic if the formula $\beta$ is atomic.

	Here is a simple example: let $S$ be an external term and  $t_1:\alpha_1,  t_2:\alpha_2$ be internal labelled formulas. As we said above, the expression $(t_1:\alpha_1)\land (t_2:\alpha_2)$ is not an internal labelled formula, but $S: (t_1:\alpha_1)\land (t_2:\alpha_2)$ is an external labelled one. Even if we have not  defined the semantics yet,  to help the reader's intuition we anticipate that the latter formula will be true in a structure  at the external time instant ``$S$'' if, in the internal time flow corresponding to ``$S$'', formulas $\alpha_1$ and $\alpha_2$ are true at time instants ``$t_1$'' and ``$t_2$'' respectively.
	
	In the following we will deal with \emph{sequents}. As is usual, a  sequent is a formal expression of the form   $A_1,\dots,A_m \vdash B_1,\dots,B_n,$ for some $m,n\in\omega,$ where  $A_1,\dots, A_m$ and  $B_1,\dots,B_n$ are lists of formulas in $F_{e,r}\cup F_{e,l}.$ 
	
	\subsection{Semantics}\label{semantics}
	
	The notation introduced in section~\ref{syntax} is in force. In the following we use  standard model-theoretic notation (see, for instance, \cite{keisler}).
	
	A \textit{pre-structure} is a  a tuple of the form $(\mathbf  M, (\mathbf  N_a)_{a\in A}))$  such that:
	
	\begin{itemize}
		\item[-] $\mathbf  M=(M , < , 0, g: M\rightarrow M),$ where  $(M , < , 0)$ a dense linear ordering with least element 0 but no greatest element and $g$ a strictly increasing function such  that, for all $a\in M,$ $a<g(a)$ and the sequence $\{g^n(0):  n\in\omega\}$ is cofinal in $M$  (namely, for all $a\in M$ there exists $n\in\omega$ such that $a<g^n(0)$); 
		\item[-] for each $a\in M,$ $\mathbf  N_a=(N_a , <_a , 0_a, g_a: N_a\rightarrow N_a)$  is a first-order structure that satisfies   analogous properties to those of  $\mathbf  M.$ 
	\end{itemize}
	For notational simplicity, we will write $(\mathbf  M, \mathbf  N_a)_{a\in A}$ for $(\mathbf  M, (\mathbf  N_a)_{a\in A})).$
	
	In order to assign a truth value to a formula (internal or external) we fix 
	\begin{enumerate}
		\item[-] an assignment $\sigma: V_e \rightarrow M,$ {  where we denote by} $V_e$  the set of $L_{e,r}$-variables;
		\item[-] a family of assignments $\{\sigma_a: V_i \rightarrow N_a\}_{a\in M}, $ where {  where we denote by} $V_i$  the set of $L_{i,r}$-variables;
		\item[-] a family $\{p^{\mathbf  N_a}: p\in \mathbf P\}$ of subsets of $N_a,$ for each $a\in M.$ The set $p^{\mathbf  N_a}$ is intended to be the set of points in $N_a$ where the propositional letter $p$ is true.
		
	\end{enumerate}

	We write $t^{\mathbf  N_a,\sigma_a}$ to denote the interpretation of  $t\in T_i$  in $\mathbf  N_a$ under  $\sigma_a.$  Let $b\in N_a.$ We denote by $\sigma_a(x/b)$ the assignment that differs from $\sigma_a$ at most on the variable $x,$ on which it takes value $b.$

	The truth value of  an $L_{i,r}$-formula in $\mathbf  N_a$ under $\sigma_a$ is given by its first-order semantics. 
	
	Let $t\in T_i$ and let  $\psi\in F_{i,t}.$   We fix $a\in M.$  Let   $\overline b=t^{\mathbf  N_a,\sigma_a}.$  We recursively define  $\mathbf  N_a,\sigma_a,(p^{\mathbf  N_a})_{p\in\mathbf P}\models t:\psi$ as follows.
	
	\begin{enumerate}
		\item[-] $\mathbf  N_a,\sigma_a,(p^{\mathbf  N_a})_{p\in\mathbf P}\models t: p$\  if\  $\overline b\in p^{\mathbf  N_a};$
		\item[-] the cases relative to the propositional connectives are treated as expected;
		\item[-] $\mathbf  N_a,\sigma_a,(p^{\mathbf  N_a})_{p\in\mathbf P}\models t: \Box_{[m,n]}\psi$ if,  for all $d\in N_a$ such that $(g_a)^m(\overline b)\le_a d\le_a (g_a)^n(\overline b),$ it holds that ${\mathbf  N_a},\sigma_a(x/d),(p^{\mathbf  N_a})_{p\in\mathbf P}\models x: \psi.$

		\item[-] $\mathbf  N_a,\sigma_a,(p^{\mathbf  N_a})_{p\in\mathbf P}\models t:\Box_{]m,\infty[}\alpha$ if, for all $d\in N_a$ such that $g_a^m(\overline b)<_a d,$ it holds that ${\mathbf  N_a},\sigma_a(x/d),(p^{\mathbf  N_a})_{p\in\mathbf P}\models x: \psi.$
		
		\noindent (The  cases { relative to the other temporal operators} are similar.)

		\end {enumerate}
		
		We let $\mathcal M =(\mathbf  M, \mathbf  N_a, (p^{\mathbf  N_a})_{p\in\mathbf P},\sigma, \sigma_a)_{a\in M}$ and, with a little abuse of the model-theoretic terminology,  we call $\mathcal M$  a \textit{structure}. We say that $\mathcal M$ is based on the pre-structure $(\mathbf  M, \mathbf  N_a)_{a\in M}.$
		
		As with the internal relational formulas, the truth value of an $L_{e,r}$-formula $\alpha$ in $\mathbf  M$ under $\sigma$ is given by by its first-order semantics. Relative to such   an $\alpha,$ we will often write $\mathcal M\models \alpha$ instead of  $\mathbf  M, \sigma\models \alpha,$ since there is no ambiguity.
		
		Relative to $\mathbf M, \sigma$ we will use the same notation previously introduced for $\mathbf N_a,\sigma_a,$ $a\in M.$ Let $T\in T_e$ and  let  $\overline a\in M$  be the interpretation $T^{\mathbf  M, \sigma}$ of $T$ in $\mathbf M$ under $\sigma.$

		If  $\varphi\in F_{i,r},$  then $\mathcal M\models T:\varphi$
		if $\mathbf  N_{\overline a},\sigma_{\overline a}\models\varphi.$  
		
		Next we say when  $\mathcal M\models T: \varphi, $ for  $\varphi\in\overline F_{i,l}.$ Basically we will repeat (with the necessary changes) the definition that we gave in the case of the  internal formulas. We provide some  details for sake of clarity. Recall that the ``atomic'' formulas in $\overline F_{i,l}$ are those of the form $t: \psi,$ where $t\in T_i$ and $\psi\in F_{i,t}.$
		\begin{itemize}
			\item[-] $\mathcal M\models T: t: \psi$  if $\mathbf  N_{\overline a}\models t: \psi;$
			\item[-] the propositional cases are treated as expected;
			\item[-] $\mathcal M\models T: \Box_{[m,n]}\eta$ if,  for all $d\in M$ such that $g^m(\overline a)\le d\le g^n(\overline a),$ it holds that $(\mathbf  M, \mathbf  N_a,(p^{\mathbf  N_a})_{p\in\mathbf P},\sigma(X/d), \sigma_a)_{a\in M}\models X: \eta.$
			
			\noindent (The other cases are similar.)
			
		\end{itemize}
		
		We say that a  sequent $\Gamma\vdash\Delta$    is true, or satisfied,  in  a structure $\mathcal M$ (and we write $\mathcal M\models \Gamma\vdash\Delta$)  if  $\mathcal M$ does not satisfy some of the formulas in  $\Gamma$ or satisfies  some of the formulas in $\Delta.$ 
		A sequent is \textit{valid}  if it is true in all structures.

		\section{Sequent calculus}\label{cal}

		In this section we introduce the  axioms and the rules   of  $\mbox{MTL}_\infty^2.$  Most of the axioms are formulated as Post rules.
		
		In the following we first introduce some axiom schemas, then we will consider instances of those schemas. 
		We assume that the reader can easily figure out what we mean by an instance of some axiom schema.  For instance, for all $S_1, S_2$ in $T_e,$ the sequent  $S_1=S_2, \rho[S_1/Z]\vdash \rho[S_2/Z]$ is an instance of  the axiom schema $X=Y, \rho[X/Z]\vdash \rho[Y/Z]$.
		
		The reason for dealing with instances rather than with schemas is to obtain a suitable cut-elimination theorem. The reader who is not interested in proof-theoretic issues may opt for axiom schemas.

		\subsection*{Axioms for equality}
		
		Among the  axioms for equality are all instances of the following axiom schemas. 
		
		$$
		\vdash X=X; \hspace{1cm} X=Y, \rho[X/Z]\vdash \rho[Y/Z] \quad (\rho \mbox{\ an atomic formula in\ } F_{e,r}) 
		$$
		$$
		\vdash T: x=x; \hspace{1cm} T: x=y, T: \rho[x/z]\vdash T: \rho[y/z] 
		$$
		\begin{flushright}
			$(T\in T_e,\  \rho \mbox{\ an atomic formula in\ }  F_{i,r})$
		\end{flushright}

		In addition, we have the following equality  axioms:
		
		$$ T=S,\  T:\eta \vdash S: \eta \qquad \mbox{($\eta$ an atomic formula in  $\overline F_{i,l}\cup F_{i,r}$} ) $$

		As it can be immediately realized, the axiom schemas  on the second line are obtained from  those on the first one by using internal variables/symbols and by ``decorating'' with a same external label all the formulas involved. Loosely speaking, we will say  that the latter schemas  are the \textit{internal} formulations of the former.  Later we will adopt the same terminology relative to the rules.
		
		\subsection*{Extralogical axioms}
		
		With the extralogical axioms we axiomatize the class of structures under consideration (see Section~\ref{semantics}).
		
		We start with  external extralogical axiom schemas stating that $<$ is a dense linear ordering with least element $C$:
		
		$$X<X\vdash\hspace{3mm} ;\quad X<Y, Y<Z\vdash X<Z;\quad \vdash X<Y, X=Y, Y<X $$   
		$$\vdash C< X, C=X;\quad X<Y\vdash \exists Z (X<Z\land Z<Y) $$ 
		
		We do not include an axiom stating that there is no greatest element, because the latter property can be derived from the leftmost axiom of the following list.

		The external extralogical axioms  also express  the properties of the function symbol $F$ and of the sequence $\{F^n(C): n\in\omega\}$:
		$$\vdash X< F(X); \quad X<Y \vdash F(X)< F(Y); \quad \vdash\bigvee_{n\in\omega}(X<F^n(C))$$

		The external extralogical axioms are all instances of the above schemas.
		
		As with the axioms for equality, we have internal versions of all the external extralogical axioms.  We leave to the reader the easy task of formulating them.
		
		\medskip
		
		In the following we refer to \cite[Chapter 2]{ptlc} for a formulation of the rules of first-order sequent calculus.

		\subsection*{Identity axiom and cut  rule}
		
		They  are:$$
		A\vdash A\quad \textit{id} \hspace{2cm}
		\brule %
		{\Gamma\vdash A,\Delta} %
		{ \Gamma,A \vdash \Delta} %
		{\Gamma\vdash \Delta} %
		{\quad \textit{cut} }
		$$%
		respectively. They apply to   any formula $A\in F_{e,r}\cup  F_{e,l}. $

		\subsection*{Structural rules}
		
		The structural rules are those of  \textit{Weakening, Exchange} and \textit{Contraction} from sequent calculus. They apply to any formula in $F_{e,r}\cup  F_{e,l}.$

		\subsection*{Rules for relational formulas}
		The rules for the external relational formulas are those of the  first-order   sequent calculus for the finitary connectives and the
		quantifiers. In addition, we have the following  rules for the  infinitary disjunction:

		\begin{center}
			$\urule %
			{\{\Gamma,T\lt F^n(C)\vdash\Delta\}_{n\in\omega} } %
			{\Gamma,\bigvee_{n\in\omega}(T\lt  F^n(C))\vdash\Delta } %
			{} $ %
			\qquad
			$\urule %
			{\Gamma\vdash T\lt F^m(C), \Delta}%
			{\Gamma\vdash\bigvee_{n\in\omega}(T\lt F^n(C)),\Delta } %
			{\quad\mbox{for all  $m\in\omega.$}} $ %
		\end{center}
		
		We also have internal versions of all the rules above for the labelled external formulas of the form $T: \rho,$ where $T\in T_e$ and $\rho\in F_{i,r}.$ For sake of clarity, we formulate the rules concerning  the infinitary disjunction:

		\begin{center}
			$\urule %
			{\{\Gamma,T:t\lt f^n(c)\vdash\Delta\}_{n\in\omega} } %
			{\Gamma,T:\bigvee_{n\in\omega}(t\lt  f^n(c))\vdash\Delta } %
			{} $ %
			\qquad
			$\urule %
			{\Gamma\vdash T:t\lt f^m(c), \Delta}%
			{\Gamma\vdash T:\bigvee_{n\in\omega}(t\lt f^n(c)),\Delta } %
			{\quad\mbox{for all  $m\in\omega.$}} $ %
		\end{center}

		\subsubsection*{Propositional rules }
		
		The propositional rules for the external labelled formulas closely follow the standard propositional
		rules.   As an example, we give the the rules for implication:
		
		%
		%
		
		\begin{center}
			
			$\urule %
			{\Gamma, T:\alpha\vdash T:\beta, \Delta}%
			{\Gamma\vdash T:\alpha\rightarrow \beta, \Delta } %
			{} $ %
			\hspace{ 3mm}
			$\urule %
			{\Gamma,T: \beta\vdash \Delta\qquad\Gamma\vdash T:\alpha,\Delta}%
			{\Gamma, T:\alpha\rightarrow \beta\vdash \Delta } %
			{} 
			\hspace{2mm}
			\ (\alpha,\beta\in \overline F_{i,l} \mbox{\ or\ } \alpha,\beta\in F_{i,r})$
		\end{center}
		
		Notice the constraint on the application of the two rules above, which is a consequence of the constraint that we put on the formation of the set $F_{e,l}$ of external labelled formulas. As we already said, we do not want to ``mix up'' the relational and temporal components of the system.
		
		\medskip
		
		We also have internal versions of all the propositional rules for the external labelled formulas. For instance, the internal versions of the last two rules are the following: 
		
		\begin{center}
			
			$\urule %
			{\Gamma, T:t:\eta\vdash T:t:\xi, \Delta}%
			{\Gamma\vdash T:t:\eta\rightarrow\xi, \Delta } %
			{} $ %
			\qquad$\urule %
			{\Gamma,T: t:\xi\vdash \Delta\qquad\Gamma\vdash T:t:\eta,\Delta}%
			{\Gamma, T: t:\eta\rightarrow\xi\vdash \Delta } %
			{}  
			\quad(\eta,\xi\in F_{i,t})$
			
		\end{center}

		\subsection*{Temporal rules}
		The rules for temporal operators are essentially the same as  in \cite{bm}.
		As above, we have a ``duplication'' of the rules to take into account the internal and the external level of the deductive system. For instance, the rules governing $\Box_{[m,n]}$ are: 
		\begin{center}
			$\urule %
			{\Gamma, F^m(T)\le X, X\le F^n(T)\vdash X:\alpha,\Delta}%
			{\Gamma\vdash  T:\Box_{[m,n]}\alpha,\Delta } %
			{\qquad\mbox{($\alpha\in\overline F_{i,l};$ $X$ not free in $\Gamma\cup\Delta\cup \{ T\}$) }}
			$ %
			
			\bigskip 
			
			$\urule %
			{\Gamma\vdash F^m(T)\le S, \Delta \quad \Gamma\vdash S\le F^n(T), \Delta\quad
				\Gamma, S:\alpha\vdash\Delta}%
			{\Gamma,  T:\Box_{[m,n]} \alpha\vdash\Delta } %
			{}  
			\qquad (\alpha\in\overline F_{i,l})$ %
			
		\end{center}

		\begin{center}
			$\urule %
			{\Gamma, T: f^m(t)\le x, T:x\le f^n(t)\vdash T:x:\eta,\Delta}%
			{\Gamma\vdash  T:t:\Box_{[m,n]}\mathfrak{\eta},\Delta } %
			{\quad\mbox{($\eta\in F_{i,t};$ $x$ not free in $\Gamma\cup\Delta\cup\{t\}$}) } $ %
			
			\bigskip 
			
			$\urule %
			{\Gamma\vdash T: f^m(t)\le s, \Delta \quad \Gamma\vdash T: s\le f^n(t), \Delta \quad
				\Gamma, T: s:\eta\vdash\Delta}%
			{\Gamma, T: t:\Box_{[m,n]}\eta\vdash\Delta } %
			{}  
			\qquad(\eta\in F_{i,t})$ %
			
		\end{center}

		%
		%
		%
		%
		%
		%
		
		The rules for the other temporal operators are similar.  They can be easily derived from the  rules above.

		Proofs are standardly defined.   Notice that,
		although each branch in a proof  is finite, the proof itself may  have infinite
		height because of  the infinitary rules.
		
		\medskip
		
		It may help the reader's intuition to keep in mind that,  from the point of view of the internal logic, the atomic temporal formulas are those  in the set $\mathbf P$ of propositional letters.  The latter  set gets closed under (internal) boolean connectives and (internal) modal operators   to form the set $F_{i,t}$ of internal temporal formulas.  
		
		From the point of view of the external logic, the formulas  in the set $F_{i,l}$ of internal labelled formulas play the role of atomic temporal formulas.   By taking  the closure $\overline  F_{i,l}$ of $F_{i,l}$ under (external) boolean connectives and (external) modal operators we get   the external temporal formulas.
		
		Forgetting about  the labels, the deduction rules governing  $F_{i,t}$ are those of propositional sequent calculus and those for the (internal) temporal operators. Similarly, the deduction rules governing $\overline  F_{i,l}$ are those of propositional sequent calculus and those for the (external) temporal operators. Basically, we have two copies of the same sequent calculus, one for each level.   
		
		Each  relational component is governed by the deduction rules of the first-order sequent calculus, plus   the corresponding infinitary rule for disjunction.
		
		As is usual, we denote by $\vdash$ the provability relation in $\mbox{MTL}_\infty^2.$
		
		\medskip

		\medskip
		
		We end this section by  remarking that  two induction schemas are provable in $\mbox{MTL}_\infty^2.$  More precisely, let  $T\in T_e$ and $\alpha\in\overline F_{i,l}.$  Then

		$$\vdash
		T:(\alpha\land\Box_{[0,\infty[}(\alpha\rightarrow\Box_{[0,1]}\alpha))\rightarrow\Box_{[0,\infty[}\alpha\qquad (1)$$ 
		
		is provable in $\rm{MTL}_\infty^2.$ Up to  typographic changes, the proof is the same as the induction schema  in \cite{bm}. In a similar  way, for every  $T\in T_e, t\in T_i$ and $\eta\in F_{i,t},$  we have 
		
		$$\vdash
		T: t:(\eta\land\Box_{[0,\infty[}(\eta\rightarrow\Box_{[0,1]}\eta))\rightarrow\Box_{[0,\infty[}\eta\qquad (2)$$ 
		
		We refer the reader to \cite{bm} for a remark on the unprovability of (1) when the   infinitary axiom $\vdash\bigvee_{n\in\omega}(X<F^n(C))$ gets removed. The example given therein can be easily adapted to show the unprovability of (2) without the axiom $\vdash T: \bigvee_{n\in\omega}(x<f^n(c)).$

		\section{Underlying unlabelled systems} 
		
		Some labelled logics are built on top of unlabelled systems. Even if we have deliberately chosen not to describe in full detail a possible unlabelled system underlying $\mbox{MTL}_\infty^2,$ in this subsection we provide some insight about the semantics of the latter. We are reluctant to call the resulting object a logic. For it should be clear that, for the purposes of this paper, a logic is made of two components: a semantic and a deductive one. Indeed the reader can easily convince himself that  deduction rules  for the temporal operators corresponding to those introduced in Section~\ref{cal} cannot be formulated in absence of labels, even in case of the one-dimensional system  $\mbox{MTL}_\infty$ introduced in \cite{bm}.
		
		For sake of simplicity, let us  consider  $\mbox{MTL}_\infty.$ The  formation rules for the underlying unlabelled formulas are those of propositional first-order logic expanded with those for the temporal operators $\Box_{(m,n)},$ with $m\in\omega$ and $n\in\omega\cup\{\infty\}.$  More precisely: if $\alpha$ is a formula then so is $\Box_{(m,n)}\alpha,$ for all possible choices of the operator $\Box_{(m,n)}.$ 
		
		The pre-structures which are adequate to interpret those unlabelled formulas are of the form 
		$\mathbf  M=(M , < , 0, g: M\rightarrow M),$  where  $(M , < , 0)$ is a dense linear ordering with least element 0 but no greatest element and $g$ a strictly increasing function such  that, for all $a\in M,$ $a<g(a)$ and the sequence $\{g^n(0):  n\in\omega\}$ is cofinal in $M.$  
		In order to get a structure we must provide  a family $\{ p^{\mathbf M}: p\in \mathbf P\}$ of subsets of $M$. We let  
		\begin{equation}\mathcal M=(\mathbf M, (p^{\mathbf M})_{p\in\mathbf P})\label{one}
		\end{equation}
		
		The semantics is Kripke-like. Let $\mathcal M$ be as in (\ref{one}) and  $m\in M.$ The definition of $\mathcal M,m\models \alpha$ is by induction on $\alpha.$ We skip the the details (see Section~\ref{sinsem} or \cite{bm} for insights). 
		
		We stipulate that an unlabelled formula $\alpha$ is \textit{valid} if, for all $\mathcal M$ as in (\ref{one}) and all $m\in M,$ $\mathcal M, m\models\alpha.$ 
		
		We make the trivial observation  that, for $\mathcal M$ as (\ref{one}) and $m\in M,$ by letting $M_m=\{ x\in M: m\le x\}$  and by denoting with the same names the restrictions of  by $<$ and $g$  to $M_m$ respectively,  the tuple $\mathbf M_m=(M_m, <, m, g)$  together with the family 
		$\{p^{\mathbf M}\cap M_m: p\in\mathbf P\}$ forms a structure as in (\ref{one}). It follows that  validity of $\alpha$ is equivalent to the following:  $$\mbox{for all $\mathcal M$ as in (\ref{one}),}\hspace{7mm} \mathcal M, 0\models \alpha.$$ 
		
		Eventually, we notice that the latter is equivalent to the validity of the labelled formula $c:\alpha$ with respect to the  $\mbox{MTL}_\infty$ semantics. By soundness and completeness of $\mbox{MTL}_\infty$  (see \cite[Theorem 3.4]{bm}), validity of the unlabelled formula $\alpha$ is equivalent to the provability in $\mbox{MTL}_\infty$ of the labelled sequent  $\vdash c:\alpha.$
		
		Therefore we may use the $\mbox{MTL}_\infty$  calculus to investigate validity in the underlying unlabelled system. Admittedly, this  is of limited interest because of the above mentioned fact that the unlabelled system  lacks deduction rules for its temporal operators. As such, it is not a logic whose validity or provability relation we may want to investigate within another logic.
		
		In light of soundness and completeness  (see Corollary~\ref{completeness}), we may proceed as above in the case of $\mbox{MTL}_\infty^2$ In this latter case, the underlying system should  have two families of temporal operators $\Box^i_{(m,n)}$ and $\Box^e_{(m,n)},$ with $m\in\omega$ and $n\in\omega\cup\{\infty\}$ and with the obvious intuitive meaning of the superscripts \textit{i} and \textit{e}. We should also impose, among others,  syntactic restrictions on the formation rules of formulas. For instance, no  $\Box^e_{(m,n)}$ can occur on the right-hand side of a $\Box^i_{(m,n)}$ operator. The structures which are adequate to interpret the resulting class of formulas can be easily  defined, by closely following Section~\ref{semantics}.
		
		As in the case of $\mbox{MTL}_\infty,$ it is possibile to relate validity in the unlabelled system to validity  with respect to the
		$\mbox{MTL}_\infty^2$ semantics. We skip the details and we provide a simple example:  validity of the unlabelled formula $\Box^e_{[0,1]}\Box^i_{[0,1]}p$ turns out to be equivalent to the validity  of the sequent $\vdash C: \Box_{[0,1]} c: \Box_{[0,1]}p.$ 
		
		The same considerations as in the one-dimensional case apply. Unfortunately, even in this case the unlabelled system   lacks deduction rules for the temporal operators.

		\section{Reduction tree construction and completeness}\label{compl}
		
		At the end of Section~\ref{semantics} we already said when a sequent is valid. It is a lengthy but easy task to verify that all the axioms are valid and that every provable sequence is valid. To prove the latter, proceed by induction on a derivation of a provable sequent.
		
		We aim at proving that $\rm{MTL}_\infty^2$ is complete with respect to the class of structures that we introduced in Section~\ref{semantics}.
		
		\subsection{The reduction tree}\label{redtree} 
		
		
		\newcommand{\Urule}[3]{%
			\prooftree #1 \Justifies #2 \using #3 \endprooftree}

		\newcommand{\Exp}[2]{%
			#1\uparrow #2
		}
		
		\newcommand{\rExp}[2]{%
			#1\Uparrow #2
		}

		In this section we suitably modify the reduction tree construction technique of  first order logic  (see, for instance, Theorem 3.1.9, Remarks 3.1.11 and 3.1.12 in \cite[\S 3.1]{ptlc}). Our aim is to get either a proof in $\rm{MTL}^2_\infty$ of a sequent $\Gamma\vdash\Delta$ or a tree with an infinite branch whose root is labelled $\Gamma\vdash\Delta.$    In Section~\ref{countermodel} we will derive the unprovability  of  $\Gamma\vdash\Delta$ from the existence of such an infinite branch.
		
		For simplicity we identify a node in a tree with its associated label.

		We fix a sequent $\Gamma\vdash\Delta.$ We also fix 
		\begin{enumerate}
			\item an ordering $A_1,\ldots,A_m$ of all the formulas in $\Gamma\vdash\Delta$  (recall that all the $A_i$'s belong to $F_{e,r}\cup F_{e,l}$);
			\item an enumeration   of all  axioms different from \textit{id}, in which each axiom is repeated infinitely many times;
			\item an enumeration  ${\{{t}_i}\}_{i<\omega}$  of the set $T_i$ of $L_{i,r}$-terms;
			\item an enumeration  ${\{{S}_i}\}_{i<\omega}$  of the set $T_e$ of $L_{e,r}$-terms.
		\end{enumerate}
		
		Next we  recursively define a sequence $\{\TR_n\}_{n<\omega}$  of well founded trees.  We say that a branch in $\TR_n$ is \emph{closed} if its leaf  is an axiom  (of any kind), otherwise it is called \emph{open}.  It is intended that the following construction only applies to the open branches in a tree. Therefore, if for some $n$ all  branches in $\TR_n$ are closed, then $\TR_n=\TR_m$ for all  $n<m.$
		
		\begin{description}
			\item[basis] $\TR_0$ is $\Gamma\vdash\Delta$.
			\item[recursion] Assume that the trees $\TR_m,$ $m\le n,$ have been already defined in a way that, for all $h<k\le m,$  $\TR_k$ is an upwards extension of $\TR_h.$  If all the branches in $\TR_n$ are closed, we stop. If not,   we define  $\TR_{n+1}$ by  extending each of the open branches in $\TR_n$ upwards. With each open branch we proceed as follows.
			
			
			Let  $\Gamma^n\vdash\Delta^n$  be the leaf of an open branch in $\TR_n.$ \ The formulas in $\Gamma^n\vdash\Delta^n$ come equipped with an ordering $A_1, \ldots,A_q$ that has been determined by the previous steps of the construction (see cases 1--3 below).

			We examine various cases, depending on the first formula $A_1$  in the ordering.
			
			In the  following, we write $
			\Urule{S_1}{S_2}{}$ if the sequent $S_2$ can be proved from the sequent $S_1$ by  application of finitely many structural rules.

			
			\begin{enumerate}
				\item\label{1one}  $\Gamma^n\vdash\Delta^n$ is provable by means of structural rules only from some   axiom  (of any kind), say from $ \Gamma'\vdash\Delta'$. In this case we extend $\Gamma^n\vdash\Delta^n$ upwards with the corresponding bit of derivation. 
				
				$$
				\Urule{\Gamma'\vdash\Delta'}{\Gamma^n\vdash\Delta^n}{}.
				$$

				Notice that the branch so obtained  is closed. Therefore it will not be further extended.
				\item\label{extra-logical-expansion} Case~\ref{1one} does not happen  and $A_1$ is atomic.

				In this case we examine the  $n$-th element  in the enumeration of the axioms. Let $B_1,\ldots,B_k\vdash D_1,\ldots,D_m$ be such element.  
				
				If  $B_1,\ldots,B_k$  all occur in $\Gamma^n,$  we   extend $\Gamma^n\vdash\Delta^n$ upwards as follows:
				
				\def\unobis{
					\ded{B_1,\ldots,B_k\vdash D_1,\ldots,D_m}
					{\Gamma^n\vdash D_1,\ldots,D_m, \Delta^n}
					{\qquad\mbox{\scriptsize{structural rules}}}
				}
				
				\def\uno{
					\Urule{B_1,\ldots,B_k\vdash D_1,\ldots,D_m}
					{
						{\Gamma^n\vdash D_1,\ldots,D_m, \Delta^n }
					}
					{}
				}
				
				\def\due{
					\urule{\uno  \Gamma^n, D_1\vdash\Delta^n}
					{\Gamma^n\vdash D_2,\ldots,D_m, \Delta^n} 
					{cut}
				}
				\def\tre{
					\urule{\due  \Gamma^n, D_2\vdash\Delta^n}
					{\vdots \ cuts}
					{{cut}}
				}
				\[
				\urule{\tre \hspace{-3ex} \Gamma^n, D_m\vdash\Delta^n}
				{\Gamma^n\vdash\Delta^n}
				{cut}
				\]
				For each $1\le i\le m,$ the  ordering of the formulas in the leaf $\Gamma^n, D_i\vdash\Delta^n$ is  $A_2,\ldots,A_q, D_i, A_1.$ 
				
				If  some among $B_1,\ldots,B_k$  does not occur in $\Gamma^n,$  we do not add any node to the current  branch and we re-order the formulas in $\Gamma^n\vdash\Delta^n$ as follows:   $A_2,\dots, A_q, A_1.$

				\item\label{two} Case~\ref{1one} does not happen and $A_1$ is not atomic. We consider several cases and subcases (one for each logical rule of the calculus).  In the following we deal with few cases only.  We invite the reader to work out the missing cases.
				\begin{enumerate}
					\item $A_1$ occurs in $\Gamma^n.$
					\begin{enumerate}
						\item[(a1)] $A_1$ is   $T: t:\eta\rightarrow\xi.$ We extend  $\Gamma^n\vdash\Delta^n$ upwards as follows:
						\[
						\Urule %
						{\Gamma^n,T: t:\xi\vdash \Delta^n\qquad\Gamma^n\vdash T:t:\eta,\Delta^n}%
						{\Gamma^n\vdash \Delta^n } %
						{}   
						\]
						The new orderings of the formulas in the left  and the right  leaf are   $A_2,\ldots,A_q,T: t:\xi, A_1$ and $A_2,\ldots,A_q,T: t:\eta, A_1$ respectively.
						\item[(a2)] $A_1$ is   $T: t:\Box_{[i,j]}\eta.$ Let $s$ be the first term in the enumeration ${\{{t}_i}\}_{i<\omega}$ that has not been  used in the current branch in a previous application of this case to the formula $A_1$.
						We extend
						$\Gamma^n\vdash\Delta^n$ upwards as follows
						\[
						\Urule %
						{\Gamma^n\vdash T: f^i(t)\le s,\Delta^n\quad \Gamma^n\vdash T: s\le f^j(t),\Delta^n\quad
							\Gamma^n, T: s:\eta\vdash\Delta^n}%
						{\Gamma^n\vdash\Delta^n} %
						{}  
						\]
						The  orderings  of the formulas in the new leaves are (from left to right): $$A_2,\ldots,A_q, T: f^i(t)\le s, A_1\hspace{1cm}
						A_2,\ldots,A_q,T: s\le f^j(t), A_1$$ and 
						$A_2,\ldots,A_q,T: s:\eta, A_1.$
						\item[(a3)] 
						$A_1$  is  $T:\bigvee_{i\in\omega}(t\lt  f^i(c)).$ We 
						extend $\Gamma^n\vdash\Delta^n$ upwards as follows
						\[
						\Urule %
						{\{\Gamma^n,T:t\lt f^i(c)\vdash\Delta^n\}_{i\in\omega} } %
						{\Gamma^n\vdash\Delta^n } %
						{} 
						\]
						The  ordering of the formulas in the leaf $\Gamma^n,T:t\lt f^i(c)\vdash\Delta^n$ is  $A_2,\ldots,A_q, T:t\lt f^i(c), A_1, $ for all $i\in\omega.$\end{enumerate}
					\item $A_1$ occurs in $\Delta^n.$
					\begin{enumerate}
						\item[(b1)] if $A_1$ is  $T: t:\eta\rightarrow\xi. $  We 
						extend $\Gamma^n\vdash\Delta^n$ upwards as follows
						
						\[
						\Urule %
						{\Gamma^n, T:t:\eta\vdash T:t:\xi, \Delta^n}%
						{\Gamma^n\vdash \Delta^n } %
						{} 
						\]
						The  ordering of the formulas in the new leaf  is:\\ 
						$A_2,\ldots,A_q,T: t:\eta, T: t:\xi,A_1$
						\item[(b2)] $A_1$ is $ T:t:\Box_{[i,j]}\mathfrak{\eta}.$  We pick an internal variable $x$ that does not occur free in $\Gamma^n\cup\Delta^n\cup\{t\}$ and we 
						extend $\Gamma^n\vdash\Delta^n$ upwards as follows
						\[
						\Urule %
						{\Gamma^n, T: f^i(t)\le x, T:x\le f^j(t)\vdash T:x:\eta,\Delta^n}%
						{\Gamma^n\vdash\Delta^n }
						{} 
						\] %
						The  ordering of the formulas in  the new leaf is\\ 
						$A_2,\ldots,A_q,T: f^i(t)\le x, T:x\le f^j(t), T: x:\eta,A_1$\\
						\item[(b3)] $A_1$ is  $T:\bigvee_{i\in\omega}(t\lt f^i(c)).$   Let $m$ be the least natural number $k$ such that  $T:t\lt f^k(c)$ does not occur in $\Delta^n.$   We extend  
						$\Gamma^n\vdash\Delta^n$ upwards as follows
						\[
						\Urule %
						{\Gamma^n\vdash T:t\lt f^m(c), \Delta^n}%
						{\Gamma^n\vdash  \Delta^n} %
						{} 
						\]
						
						The   ordering  of the formulas in the new leaf is\\  $A_2,\ldots,A_q, T:t\lt f^m(c),A_1.$
						
					\end{enumerate}
					
				\end{enumerate}

			\end{enumerate}
		\end{description}
		
		\begin{remark}
			\label{rem} By construction, $\TR=\bigcup_{n\in\omega}\TR_n$ is a tree, called the \textit{reduction tree} of $\Gamma\vdash\Delta.$ Notice that the presence of axioms with empty premiss and the fact that  each axiom different from \textit{id} is repeated infinitely many times in the enumeration imply that any open branch keeps on being extended upwards. Therefore  if   $\TR$ has no infinite branch
			there exists  some $n\in\omega$ such that  all the branches in the tree  $\TR_n$ are closed.   It follows immediately from the construction that   $\TR=\TR_n$ yields a proof $\Gamma\vdash\Delta.$   Since all applications of the cut  rule take place in case~\ref{extra-logical-expansion}, we actually get a proof of $\Gamma\vdash\Delta$ where all  cut-formulas occur in the righthand side of some axiom  different from \textit{id}.
			
			If there is some infinite branch in $\TR, $ in Section~\ref{countermodel} we show how to get a countermodel of $\Gamma\vdash\Delta$ from such an infinite branch.
		\end{remark}


		\subsection{Completeness}\label{countermodel}
		
		Let us assume that the reduction tree $\TR$ of the sequent $\Gamma\vdash\Delta$ constructed in Section~\ref{redtree}  has  an infinite branch
		$(\Gamma_n\vdash\Delta_n)_{n\in\omega}.$  Let ${\overline \Gamma}=\bigcup\Gamma_n$
		and ${\overline \Delta}=\bigcup\Delta_n.$  
		
		We define a structure $\mathcal M=(\mathbf  M, \mathbf  N_a,(p^{\mathbf  N_a})_{p\in\mathbf P},\sigma, \sigma_a)_{a\in M}$ by using $\overline\Gamma$ and $\overline\Delta.$
		
		\medskip

		First we define $\mathbf  M.$ We define a  relation $\sim$ on the set $T_e$ of external terms  by letting $$S\sim T\ \Leftrightarrow\  (S=T)\in{\overline
			\Gamma}.$$ 
		
		The tree construction and the first order axioms for equality imply that
		$\sim$ is an equivalence relation.  Let $T^\sim$ be the $\sim$-equivalence
		class of term $T.$ We let 
		$$M=\{T^\sim:\ T\in T_e\};\ 
		F^{\mathbf  M}(T^\sim)=F(T)^\sim;\ 
		C^{\mathbf  M}=C^\sim;\ 
		S^\sim<^{\mathbf  M}T^\sim\,\Leftrightarrow\,(S<T)\in{\overline \Gamma}.$$
		
		We let $\mathbf  M=(M, <^{\mathbf  M}, C^{\mathbf  M}, F^{\mathbf  M}).$ (Notice the small notational abuse.)
		The axioms for equality and the reduction tree construction imply that everything is well-defined.  Furthermore, the extralogical axioms and the reduction tree construction imply
		that $\mathbf  M$ satisfies the required properties. Namely,  $(M, <^{\mathbf  M}, C^{\mathbf  M})$ is a dense linear ordering with least element $C^{\mathbf  M}$ but no greatest element; the function $F^{\mathbf  M}$  is strictly increasing; for all $T^\sim\in M,$ $T^\sim<^{\mathbf  M} F(T)^\sim$ and the sequence $\{F^n(C)^\sim:  n\in\omega\}$ is cofinal in $M.$

		We define $\sigma: V_e\rightarrow M$ as follows:  $\sigma(X)=X^\sim.$ By induction on external  terms it follows that  the
		interpretation $T^{{\mathbf 
				M},\sigma}$ of term $T$ in $\mathbf  M$ under $\sigma$ is just $T^\sim,$ for
		each  $T\in T_e.$
		
		\medskip
		
		Next we  basically repeat   the construction of $\mathbf  M$ to get $\mathbf  N_a,$ for each $a\in M.$  For each $S\in T_e,$ we define an equivalence relation $\sim_S$ on $T_i$ as follows:
		$$t_1\sim_S t_2\ \Leftrightarrow\  (S: t_1=t_2)\in{\overline
			\Gamma}.$$  Let $t^S$ be the $\sim_S$-equivalence class of $t\in T_i.$ For each $S\in T_e,$ we let 
		$$
		N_S=\{t^S:\ t\in T_i\};\hspace{4ex} 
		f^{{\mathbf  N}_S}(t^S)=f(t)^S;\hspace{4ex}
		c^{{\mathbf  N}_S}=c^S;
		$$
		$$
		t_1^S<^{{\mathbf  N}_S}t_2^S\,\Leftrightarrow\, (S: t_1<t_2)\in{\overline \Gamma}.
		$$
		
		Moreover, for each $p\in\mathbf P,$ we let $p^{\mathbf  N_S}=\{ t^S:  (S:t:p)\in \overline\Gamma\}.$ (Recall that we denote by $p^{\mathbf  N_S}$  the set of points in $N_S$ where $p$ is true.)
		
		As above, the reduction tree construction implies that everything is well-defined and that, for all $S\in T_e,$ $\mathbf  N_S$ satisfies  the required properties.
		
		We define $\sigma_S: V_i\rightarrow N_S$ as follows:  $\sigma(x)=x^S.$ It follows easily  that, for each $t\in T_i,$  $t^{\mathbf  N_S, \sigma_S}= t^S.$

		\medskip
		
		The careful reader may have already noticed that the support of the external time flow
		is the $\{T^\sim: T\in T_e\},$ but we defined above a family  $\{\mathbf  N_S\}_{S\in T_e}$  (instead of an indexed family  on the $\sim$-equivalence classes of $T_e,$ as required by the definition of pre-structure).  We explain why: from the axioms and the  construction of the reduction tree it follows that  $\mathbf  N_S=\mathbf  N_T$  whenever $T^\sim = S^\sim.$ 
		For instance, let us assume that $T=S$ and $T: t=s$ are both in $\Gamma^n$ and that  $T=S, T: t=s \vdash S: t=s$ is the $n$-th element in the enumeration of the axioms (one such $n$ does exist because of the infinitely many repetitions). Then, extending upwards $\Gamma^n\vdash\Delta^n$ as prescribed by case 2 of the reduction tree construction, we get  that $\Gamma^{n+1}\vdash\Delta^{n+1}$ is $\Gamma^n, S: t=s\vdash \Delta^n.$  Hence $S:t=s$ is in $\overline\Gamma.$
		Similarly, assuming  $T^\sim=S^\sim,$ we get { $\sigma_S=\sigma_T$ and $p^{\mathbf  N_S}=p^{\mathbf  N_T},$ for all $p\in\mathbf P.$ }  
		
		Hence  we simply let $\mathbf  N_{S^\sim}=\mathbf  N_T$ for some (for all) $T\in T_e$ such that $S\sim T.$ Similar considerations apply to  $p^{\mathbf  N_{S^\sim}}$ and $\sigma_{S^\sim}.$ Thus we get the structure $$\mathcal M=(\mathbf  M,\mathbf  N_{S^\sim},(p^{\mathbf  N_{S^\sim}})_{p\in\mathbf P},\sigma, \sigma_{S^\sim})_{S\in T_e}.$$
		
		However, motivated by the above considerations (and for sake of notational simplicity), in the following we identify  $\mathcal M$ with $(\mathbf  M,\mathbf  N_S,(p^{\mathbf  N_S})_{p\in\mathbf P},\sigma, \sigma_S)_{S\in T_e}.$

		\medskip
		
		Next we show that $\mathcal M$ witnesses the unprovability of the sequent $\Gamma\vdash\Delta.$ To do that we prove a stronger result.
		
		\begin{theorem}\label{cmodel} Let $\mathcal M$ be the structure  defined above.  For every formula $A\in F_{e,r}\cup F_{e,l}$  it holds that
			\begin{enumerate}
				\item[\rm(1)] $A\in {\overline \Gamma} \Rightarrow \mathcal M\models A;$
				\item[\rm(2)] $A\in {\overline \Delta} \Rightarrow \mathcal M\not\models A.$
			\end{enumerate}
		\end{theorem} 
		
		\begin{proof} We prove simultaneously (1) and (2) by induction on each of the various classes of formulas.
			\begin{enumerate}
				\item $A\in F_{r,e}.$ The deduction rules governing $A$ are the rules for relational formulas (see Section ~\ref{cal}).  For those rules, the construction of the reduction tree proceeds as in the first order case. Therefore   (1) and (2) can be proved  in the same way as in  the construction of the reduction tree for first order logic. The infinitary formulas are taken care of by the  cases corresponding to (a3) and (b3) in the construction of the reduction tree for the  external  relational formulas. For instance,  let us assume that  $\bigvee_{n\in\omega} (T<F^n(C))$ occurs in $\overline\Delta.$ Notice that  such formula becomes infinitely many times  the first formula in the ordering of formulas in a sequent occurring in some infinite branch. Moreover, same as in case (b3) above,  the reduction tree is constructed  by systematically adding, {  for each $n\in\omega,$  the} formula $T<F^n(C)$  to $\overline \Delta.$ Therefore, by inductive assumption, $\mathcal M\not\models T<F^n(C),$ for all $n\in\omega.$ The conclusion follows.
				
				\item\label{due2} $A$ is of the form $T: \beta,$ for some $\beta\in F_{i,r}.$  We notice that the rules governing $A$ are  ``labelled'' versions of those involved in the previous case. Basically,  we proceed as in that case to get (1) and (2).
				
				\item $A$ is of the form $T: \beta,$ for some $\beta\in \overline F_{i,l}.$ 
				
				\begin{enumerate} 
					\item\label{aaa} First we simultaneously prove (1) and (2) for $\beta$ of the form $t: \alpha,$ for some $t\in T_i$ and $\alpha\in F_{i,t}.$ We proceed by induction on $\alpha.$
					
					Since $T^{\mathbf  M,\sigma}=T^\sim$ and $t^{\mathbf  N_T,\sigma_T}=t_T,$ the  cases when $\alpha$ is some propositional letter $p$ hold by definition of $p^{\mathbf  N_T}$.
					
					The propositional cases are straightforward by the construction of the reduction tree and by inductive assumption.
					
					\medskip 
					
					We are left with the case when $\alpha$ is of the form $\Box_{[i,j]}\eta$ (the cases relative to the other temporal operators are similar).
					
					Let us assume that $T:t:\alpha$ is in $\overline\Gamma$. Let $n\in\omega$ be such that $T:t:\alpha$ is the first  in the ordering of the formulas in $\Gamma_n\vdash\Delta_n.$ (Notice that there are infinitely many such $n$'s.)  With reference  to case  3(a2) of the reduction tree construction,  we notice that:
					
					\begin{enumerate}
						
						\item[--]  if $f^i(t)_T \le_T s_T, $ namely if $T:f^i(t)\le s$ is in $\overline \Gamma, $  then $\Gamma^{n+1}\vdash\Delta^{n+1}$ cannot be the leftmost leaf; 
						
						\item[--] if $s_T\le_T f^j(t)_T, $ then, similarly to the previous case,  we get that $\Gamma^{n+1}\vdash\Delta^{n+1}$ cannot be the middle leaf.
					\end{enumerate}
					
					Therefore, for all $s$ such that $f^i(t)_T\le_T s_T \le_T f^j(t)_T,$ it holds that $T:s:\eta$ is in $\overline\Gamma.$ Hence, by inductive assumption, $\mathcal M\models T:s:\eta.$

					Since,  in 3(a2), $s$ was chosen as the first term in the enumeration $\{t_i\}_{i\in\omega}$ which was not used so far,  for all $s\in T_i$ the formula $T:s:\eta$ is in $\overline\Gamma.$ Therefore   $\mathcal M\models T:t:\alpha.$
					
					\medskip
					
					Next, we assume that $T:t:\alpha$ is in $\overline\Delta.$  Let $n\in\omega$ be such that $T:t:\alpha$ is the first  in the ordering of the formulas in $\Gamma_n\vdash\Delta_n.$  With reference  to case  3(b2) of the reduction tree construction, we have that, for a suitably chosen variable $x$,    $T: f^i(t)\le x$ and  $T:x\le f^j(t)$ are in $\overline\Gamma$ and $T:x:\eta$ is in $\overline\Delta.$ By inductive assumption and by the semantics, we get  $$\mathcal M\not\models  T: x:\eta;\quad \mathcal M\models T: f^i(t) \le x;\quad \mathcal M\models T: x \le f^j(t).$$ Therefore $\mathcal M\not \models T:t:\Box_{[m,n]}\eta.$ 
					
					\item The propositional cases are straightforward. 
					
					\item Let $\beta$ be of the form $\Box_{[i,j]}\gamma,$ for some $\gamma\in\overline F_{i,l}.$  By the same argument used in \ref{aaa}  in the case relative to the the temporal operator, we get the required conclusion. 
				\end{enumerate}
			\end{enumerate}
		\end{proof}
		
		\begin{corollary}\label{completeness}(Completeness) Every valid sequent is provable in $\rm{MTL}_\infty^2.$
		\end{corollary}

		\begin{proof} If a sequent is valid then, by Theorem~\ref{cmodel}, its reduction tree has no infinite branch. Therefore the reduction tree yields a proof of the sequent. 
		\end{proof}

		In the following we extend to pre-structures the  isomorphism  relation $\simeq$ between first order structures. We still use $\simeq$ for the extended relation.
		
		We say that two pre-structures  $(\mathbf D, \mathbf E_d)_{d\in D}$ and $(\mathbf D^\prime , \mathbf E^\prime _d)_{d\in D^\prime}$ are isomorphic  if there exists a first-order isomorphism  $g: \mathbf D\rightarrow  \mathbf D^\prime$ such that, for all $d\in D,$ $\mathbf E_d \simeq \mathbf E^\prime_{g(d)}.$
		
		\medskip
		
		Let $\mathbb Q^+$ be the set of nonnegative rationals and let $h:\mathbb Q^+\rightarrow\mathbb Q^+$ be the function defined by $h(q)=q+1.$ Let 
		$\mathbf Q=(\mathbb Q^+, <, 0, h ).$ We  call \textit{canonical}  the pre-structure $\mathbf Q^2=(\mathbf Q, \mathbf Q_r)_{r\in\mathbb Q^+},$  where, for all $r\in\mathbb Q^+,$ $\mathbf Q_r=\mathbf Q.$   
		
		Next we recall  \cite[Proposition 3.3]{bm}, which is a variant of the back-and-forth technique by means of which  the $\aleph_0$-categoricity of the theory of linear orders without endpoints is established.  The content of the above mentioned proposition is the following:  let  $\mathbf  M=(M , < , 0, g: M\rightarrow M),$ where  $(M , < , 0)$ is a countable  dense linear ordering with least element 0 but no greatest element; $g$  is a strictly increasing function such  that, for all $a\in M,$ $a<g(a)$ and the sequence $\{g^n(0):  n\in\omega\}$ is cofinal in $M.$ Then $\mathbf M$ is isomorphic to $\mathbf Q.$ 
		
		We apply \cite[Proposition 3.3]{bm} to $\mathbf  M$ and $\mathbf  N_S$ defined as in Section~\ref{redtree} and   we get $\mathbf  M \simeq \mathbf Q\simeq \mathbf  N_S,$ for all $S\in T_e.$  It follows that  $(\mathbf  M, \mathbf  N_S)_{S\in T_e}\simeq (\mathbf Q,\mathbf Q_r)_{r\in\mathbb Q^+}.$
		
		By Remark~\ref{rem}, we finally get the following strong form of  a completeness theorem:  
		
		\begin{theorem} Let $\Gamma\vdash\Delta$ be a sequent. Then exactly one of the following holds:
			\begin{enumerate}
				\item[(1)]\label{provable} there exists a proof  of $\Gamma\vdash\Delta$  in $\rm{MTL}_{\infty}^2$ {  where} each  cut--formula occurs in the righthand side of some axiom different from \textit{id};
				\item[(2)] there exists a structure $\mathcal M$ having $\mathbf Q^2$ as underlying  pre-structure  such that $\mathcal M\not\models \Gamma\rightarrow\Delta.$
			\end{enumerate}
		\end{theorem}
		In particular,  $\mbox{MTL}^2_\infty$ is complete with respect to the family of structures whose underlying pre-structure is the canonical one.

		\section{Final remarks}
		
		An  application to be pursued  in a future work is the mechanization of the ${\rm MTL}_\infty^2$ sequent calculus.  Mechanization  is a relevant issue within the class of temporal logics, in light of its strong connections with the area  of formal verification.
		There are two main approaches to mechanization: one based on model checking; the other based on proof coding  in a so called logical  framework or proof assistant (CoQ or Isabelle, for instance).
		By the nature of  the ${\rm MTL}_\infty^2$ sequent calculus, in which the deduction rules relative to each operator are independent of the others, and by the validity of a suitable form of cut-elimination, the proof coding approach seems to be the most promising.
		
		Regarding the theoretical issues, we stress that ours is a  study  of  the meaning of temporal operators in a two-dimensional time structure. The  ${\rm MTL}_\infty^2$  rules do obey  the structural proof theory paradigm that  deduction rules must syntactically  provide the ``meaning'' of each operator. It is  worth noticing that the ${\rm MTL}_\infty^2$ rules show great similarity with those for bounded  first order quantifiers. 
		
		Moreover, our  approach does not require the temporalized and the temporalizing system being the same. In particular, the underlying  time flows might have different structures or the  systems might be endowed with different temporal operators. What is crucial   is that both system do  have a reasonably good sequent calculus. In this regard, we may investigate  an intuitionistic version of ${\rm MTL}_\infty^2$ so to address issues like the the computational content of proofs; the Brower-Heyting-Kolmogorov interpretation of temporal operators; the existence of a lambda-calculus having the ${\rm MTL}_\infty^2$ formulas as types.

		Notice  that we have obtained a completeness theorem by using a syntax-driven   technique. In our opinion, a Hilbert style  formulation of ${\rm MTL}_\infty^2$  would hardly allow to achieve completeness.

		It is worth noticing that application of our  temporalization technique can be iterated in order  to obtain an $n$-dimensional system, for each $n>2.$  All the constructions and the  results obtained in this paper generalize to $n$-dimensional systems.

		We point out that  our  technique does not allow to express properties of the interaction between the two dimensions. This requires  the introduction of   new types of formulas, the definition of their semantics and an extension of the deductive system. In this way the resulting system would be closer to a full combination of logics, rather than to the temporalization of one logic by means of another. As pointed out in  \cite{fg2}, a full combination may  affect completeness or  the cut-elimination property. It may worth investigating to what extent it is possible to increase the expressive power of the resulting system while  retaining the properties of the logics to be combined. 
		
		It may also worth investigating the decidability of ${\rm MTL}_\infty^2.$  One first step would be to  establish  the decidability of the one-dimensional logic ${\rm MTL}_\infty.$ We notice that the  ${\rm MTL}_\infty$ axioms for the relational component do extend the axioms of a first order theory of dense linear orderings. The decidability of the latter theory    can be quite easily achieved by means of model-theoretic techniques.  Decidability is likely to be inherited by our extension. 
		
		In this regard,  we recall the decidability result for Metric Interval Temporal Logic that has been obtained in \cite{mitl} by reducing the satisfiability problem for that logic to a decidable problem for timed automata. The latter result refers to a subclass of our class of structures (the so called \textit{timed state sequences}), but it is possible that a similar reduction technique  applies to our setting.  
		
		Still  it is not clear  whether it is possible to combine two such completely different  decision procedures in order to establish decidability of  ${\rm MTL}_\infty,$ also in consideration of the interaction between the relational and the temporal component  of  ${\rm MTL}_\infty$  due to the presence of labels. 
		
		In conclusion, even proving decidability of the one-dimensional  ${\rm MTL}_\infty$ seems to be nontrivial. We leave it as an open problem.

	\end{document}

\end{document}